\documentclass[reqno,10pt]{amsart}
\usepackage{CJK}
\usepackage{hyperref}
\usepackage{cite}
\usepackage{amsmath}
\usepackage{mathrsfs}
\usepackage{extarrows}
\usepackage{indentfirst}
\usepackage{color}

\makeatletter
\@namedef{subjclassname@2020}{%
	\textup{2020} Mathematics Subject Classification}
\makeatother

\numberwithin{equation}{section}

\newtheorem{theorem}{Theorem}[section]
\newtheorem{lemma}[theorem]{Lemma}
\newtheorem{definition}[theorem]{Definition}

\newtheorem{proposition}[theorem]{Proposition}
\newtheorem{remark}[theorem]{Remark}
\newtheorem{corollary}[theorem]{Corollary}

\allowdisplaybreaks

\begin{document}
	
	\title[\hfil Sobolev regularity for the perturbed fractional 1-Laplacian] {Sobolev regularity for the perturbed fractional 1-Laplace equations in the subquadratic case}
	
		\author[D. Li and C. Zhang  \hfil \hfilneg]
		{Dingding Li  and Chao Zhang$^*$}
	
	\thanks{$^*$ Corresponding author.}
	
	\address{Dingding Li \hfill\break School of Mathematics, Harbin Institute of Technology, Harbin 150001, China}
	\email{a87076322@163.com}

	\address{Chao Zhang  \hfill\break School of Mathematics and Institute for Advanced Study in Mathematics, Harbin Institute of Technology, Harbin 150001, China}
	\email{czhangmath@hit.edu.cn}

	
	\maketitle
	
\begin{abstract}
This work investigates the Sobolev regularity of solutions to perturbed fractional 1-Laplace equations.
Under the assumption that weak solutions are locally bounded, we establish that the regularity properties are analogous to those observed in the superquadratic case. By introducing the threshold $\frac{p-1}{p}$, we divide the range of the parameter  $s_p$ into two distinct scenarios. Specifically, for any $s_p\in \left(0, \frac{p-1}{p}\right]$ and $q\ge p$, we demonstrate that the solutions possess $W_{\rm loc}^{\gamma, q}$-regularity for all $\gamma\in \left(0, \frac{s_p p}{p-1}\right)$ and  the $W_{\rm loc}^{1, q}$-regularity for any $s_p\in \left(\frac{p-1}{p}, 1\right)$ and $q\ge p$, respectively.
Our analysis relies on the nonlocal finite-difference quotient method combined with a Moser-type iteration scheme, which provides a systematic approach to the regularity theory for such nonlocal and singular problems.
\end{abstract}

\section{Introduction}
Let $\Omega$ be a bounded domain in $\mathbb{R}^N$ with $N\ge2$. The main focus of this paper is to study the Sobolev regularity of weak solutions to the following perturbed fractional $1$-Laplace equation
\begin{align}
	\label{1.1}
	(-\Delta_1)^{s_1}u+(-\Delta_p)^{s_p}u=0\quad\text{in }\Omega,
\end{align}
where $s_1, s_p\in (0, 1)$ and $p\in(1, 2)$. The fractional operators $(-\Delta_1)^{s_1}$ and $(-\Delta_p)^{s_p}$ are  defined by
\begin{align*}
	(-\Delta_1)^{s_1}u(x):=2\mathrm{P.V.}\int_{\mathbb{R}^N}\frac{u(x)-u(y)}{|u(x)-u(y)|}\frac{dy}{|x-y|^{N+s_1}}
\end{align*}
and
\begin{align*}
	(-\Delta_p)^{s_p}u(x):=2\mathrm{P.V.}\int_{\mathbb{R}^N}\frac{|u(x)-u(y)|^{p-2}(u(x)-u(y))}{|x-y|^{N+s_p p}}\,dy
\end{align*}
for $x\in \Omega$. Additionally, $\mathrm{P.V.}$ denotes the integral taken in the principal value sense. 

The problem \eqref{1.1} was first proposed by G\'{o}rny, Maz\'{o}n and Toledo in their recent work \cite{GMT24}, in which they investigated several phenomena of nonlocal PDEs with inhomogeneous growth within the general framework of random walk spaces. Due to the associated functionals of these $(1, p)$-Laplace equations exhibiting different growth on their respective structures, Eq. \eqref{1.1} can be regarded as a special case of $(p, q)$-growth problems. For the superlinear case $p>1$, several significant results have been established; see the pioneering studies \cite{M89, M91}, subsequent advances \cite{BM20, CMM24, M21, MR21, DM2301}, and the references therein.

Let us first focus on the classical version of problem \eqref{1.1}, namely the perturbed 1-Laplace equation
\begin{align}
	\label{1p}
	-\mathrm{div}\left( \frac{\nabla u}{|\nabla u|}\right)-\mathrm{div}\left( |\nabla u|^{p-2}\nabla u \right) =f\quad\text{in }\Omega.
\end{align}
This equation arises in both fluid mechanics \cite{DL76} and materials science \cite{S93}. Further assuming that the weak solutions to \eqref{1p} are convex, Giga and Tsubouchi \cite{GT22} proved that the gradient of solutions is continuous. Shortly thereafter, Tsubouchi removed the convexity condition and established the interior $C^1$-regularity in \cite{T24} using De Giorgi’s truncation method together with the freezing-coefficient technique. It is also worth noting that these methods were subsequently extended to show that $Du$ is continuous with respect to $(x, t)$ for the parabolic $(1,  p)$-Laplace system \cite{T25, T24A, T25A}. Related problems have also been studied in more specialized settings; see, for example, \cite{G23, MRST10} for results on the anisotropic case involving the $p$-Laplacian, where the operator exhibits linear growth in some of the coordinates. Moreover, in sharp contrast to that in the $(1, p)$-growth regime, De Filippis and Mingione in \cite{DM23} considered the functional with nearly linear growth
\begin{align}
	\label{nearly}
	\mathcal{L}(w, \Omega):=\int_{\Omega}\left[c(x)|Dw|\log(1+|Dw|)+a(x)|Dw|^q\right]\,dx, \quad q>1,
\end{align}
and established the Schauder-type estimates for its minimizers. 

%Under suitable assumptions on the coefficients, their results further yield $C^{1, \alpha}$-regularity for the corresponding minimizers.

%To the best of our knowledge, the previous results provide only interior $C^1$-regularity of weak solutions to problem \eqref{1p}. 

When addressing fractional problems involving the 1-Laplacian, several key features have been summarized in \cite[Subsection 1.1]{LZ25} as follows:

\begin{itemize}
	\item In the weak formulation, the nonlocal quotient $\frac{u(x)-u(y)}{|u(x)-u(y)|}$ is represented by a function $Z$, where $Z\in \mathrm{sgn}(u(x)-u(y))$. Here $\mathrm{sgn}(\cdot)$ denotes the set-valued sign function with $\mathrm{sgn}(0)=[-1, 1]$.
	
	\item The energy of weak solutions to elliptic nonlocal 1-Laplace equations associated with the source term $f$ can be discontinuous. Specifically, if $f\in L^\frac{N}{s}(\Omega)$ and satisfies a smallness condition, the energy of weak solutions remains bounded. However, if $\|f\|_{L^\frac{N}{s}(\Omega)}$ is sufficiently large, the energy becomes unbounded and a weak solution may fail to exist.
	
	\item Any $s$-minimal functions are continuous in the interior of the ambient domain and also continuous up to the boundary under some mild additional
	hypothesis. 
	
	\item In two dimensions, for a nonlocal variational problem arising from an image-denoising model, its minimizers preserve the same local H\"{o}lder regularity as the original image.
\end{itemize}
For more detailed results on the 1-Laplacian, see \cite{MRT16, MPRT16, AMRT08, AMRT09, MRT19, B23, BDLM23, BDLV25, NO23},
in which the distinctive properties are thoroughly discussed. These characteristics of the 1-Laplacian ultimately originate from the non-strict convexity of the corresponding space.

Next, we turn to our previous work \cite{LZ25}, where we considered the superquadratic case $p\ge 2$ and distinguished two ranges for $s_p$:
\begin{align*}
	s_p\in \left( 0,\frac{p-1}{p}\right]  \quad\text{and}\quad s_p\in \left( \frac{p-1}{p}, 1\right) .
\end{align*}
From a technical perspective, leveraging the development of the finite-difference quotient technique in the fractional setting (see \cite{BL17, BLS18, BDL2401, BDL2402, FLZ25, DKLN23, GL24}), we established the local Sobolev regularity of weak solutions to \eqref{1.1} in the supercritical case. For background on this technique, we refer the readers to \cite{L19}, which outlines its fundamental applications and historical context. It is noteworthy that the critical threshold $\frac{p-1}{p}$ characterizes the existence of $\nabla u$, differing from that of the fractional $p$-Laplacian, where the threshold is $\frac{p-2}{p}$ (see \cite{BDL2402}). Our results indicate that the involvement of the 1-structure raises this threshold.

Inspired by the aforementioned literature, we in this paper investigate the regularity of weak solutions to \eqref{1.1} in the subquadratic case. Our conclusions align closely with those in \cite{LZ25}; when the 1-structure is omitted, our results further consistent those in \cite{GL24}. However, under certain conditions,  the presence of the 1-structure can prevent the existence of $\nabla u$, in contrast to the fractional $p$-Laplacian setting, where the gradient always exists (see \cite{BDL2401}). This distinction highlights the impact of the 1-growth term on such problems. We further observe that, unlike in the superquadratic case, the local regularity estimates for weak solutions in the subquadratic regime depend on more parameters. This is primarily due to the fact that we adopt a different approach to derive the energy estimate of weak solutions, see Proposition \ref{pro9} for detailed arguments.

Before presenting our main results, we first provide the definition of a weak solution to \eqref{1.1}. The tail space $L^{p-1}_{s_pp}(\mathbb{R}^N)$ contains all function $w\in L^p_{\mathrm{loc}}(\mathbb{R}^N)$ that satisfies
\begin{align*}
	\int_{\mathbb{R}^N}\frac{|w|^{p-1}}{(1+|x|)^{N+s_pp}}\,dx<+\infty.
\end{align*}

\begin{definition}
	\label{def1}
	A function $u\in W^{s_1,1}_{\mathrm{loc}}(\Omega)\cap W^{s_p,p}_{\mathrm{loc}}(\Omega)\cap L^{p-1}_{s_pp}(\mathbb{R}^N)$ is said to be a weak solution to problem \eqref{1.1} if
	\begin{itemize}
		\item [(i)] there exists a function $Z\in L^\infty(\mathbb{R}^N\times\mathbb{R}^N)$ satisfying $Z\in \mathrm{sgn}(u(x)-u(y))$;
		\item [(ii)] for any $\varphi\in W^{s_1, 1}_{0}(\Omega)\cap W^{s_p, p}_{0}(\Omega)$, there holds
		\begin{align}
			\label{1.2}
			0&=\int_{\mathbb{R}^N}\int_{\mathbb{R}^N}Z\frac{\varphi(x)-\varphi(y)}{|x-y|^{N+s_1}}\,dxdy\nonumber\\
			&\quad+\int_{\mathbb{R}^N}\int_{\mathbb{R}^N}\frac{|u(x)-u(y)|^{p-2}(u(x)-u(y))(\varphi(x)-\varphi(y))}{|x-y|^{N+s_pp}}\,dxdy.
		\end{align}
	\end{itemize}
\end{definition}

\begin{remark}
	\label{rem0}
	No additional decay assumption is required for the 1-growth term because the integral
	\begin{align}
		\label{assume1}
		\int_{\mathbb{R}^N \setminus B_R(x_0)} \frac{dy}{|x - x_0|^{N + s_1}}
	\end{align}
	is finite.
	Moreover, the constant “1” appearing in our main results follows directly from \eqref{assume1}.
	The existence of weak solutions to problem \eqref{1.1} can be established via the method developed in \cite{B23}, and the local boundedness of weak solutions has been proved in \cite[Section~3]{LZ25}.
\end{remark}

Now, we state our main results. Define
\begin{align*}
	\mathcal{T}:=\|u\|_{L^\infty(B_{R}(x_0))}+\mathrm{Tail}(u;x_0,R),
\end{align*}
where $\operatorname{Tail}(u; x_0, R)$ denotes the tail term associated with the function $u$ at the point $x_0$ and radius $R$ (see Section~\ref{sec2} for details).

\begin{theorem}
	\label{th12}
	Let $p\in (1, 2)$ and $s_p\in \left(0, \frac{p-1}{p}\right]$. Suppose that $u$ is a locally bounded weak solution to problem \eqref{1.1}. Then, we have 
	\begin{align*}
		u \in W_{\mathrm{loc}}^{\sigma, q}(\Omega)
	\end{align*}
	for any $q\geq p$ and $s_p\leq \sigma<\frac{s_pp}{p-1}$. Moreover, there exist constants 
	$C$ and $\kappa$ depending on $N, p, q, s_1, s_p, \sigma$ 
	such that for any ball $B_R \equiv B_R\left(x_0\right) \subset \subset \Omega$ with $R\in (0,1)$ and $r\in (0,R)$,
	\begin{align*}
		[u]_{W^{\sigma,q}\left(B_r\right)}^q \leq \frac{C\left(\mathcal{T}+[u]_{W^{s_p,p}\left(B_R\right)}+1\right)^q}{(R-r)^{\kappa}}.
	\end{align*}
\end{theorem}

\begin{theorem}
	\label{th26}
	Let $p\in(1,2)$ and $s_p\in \left(\frac{p-1}{p},1\right)$. Suppose  that $u$ is a locally bounded weak solution to problem \eqref{1.1}. Then, we have 
	\begin{align*}
		u \in W_{\mathrm{loc}}^{1,q}(\Omega) ,
	\end{align*} 
	for any $q\geq p$. Moreover, there exist constants $C$ and $\kappa$ depending on $N, p, q, s_1, s_p,$ such that for 
	any ball  $B_R \equiv B_R\left(x_0\right)\subset\subset\Omega$ with $R\in (0,1)$ and any $r\in (0,R)$,
	\begin{align*}
		\|\nabla u\|_{L^q\left(B_r\right)}^q \leq \frac{C\left(\mathcal{T}+[u]_{W^{s_p, p}\left(B_R\right)}+1\right)^q}{(R-r)^{\kappa} } .
	\end{align*}
\end{theorem}

By applying the Morrey-type embedding for (fractional) Sobolev spaces and carefully selecting the relevant parameters, one can directly establish the interior Hölder regularity of weak solutions.

\begin{corollary}
	\label{cor13}
	Let $p\in (1,2)$ and $s_p\in\left( 0,\frac{p-1}{p}\right] $. Then, for any locally bounded weak solution to problem \eqref{1.1} in the sense of Definition \ref{def1}, we have
	\begin{align*}
		u\in C^{0, \gamma}_\mathrm{loc}(\Omega)
	\end{align*}
	for any $\gamma\in \left( 0,\frac{s_p p}{p-1}\right)$. Moreover, there exist constants $C$ and $\kappa$ depending on $N, p, s_1, s_p$, such that for any ball $B_R\equiv B_R(x_0)\subset\subset\Omega$ with $R\in (0,1)$ and any $r\in(0,R)$,
	\begin{align*}
		[u]_{C^{0,\gamma}(B_r)}\le \frac{C\left( \mathcal{T}+[u]_{W^{s_p,p}(B_R)}+1\right) }{(R-r)^{\kappa}}.
	\end{align*}
\end{corollary}

\begin{corollary}
	\label{cor27}
	Let $p\in (1,2)$ and $s_p\in \left( \frac{p-1}{p},1\right) $. Then, for any locally bounded weak solution to problem \eqref{1.1} in the sense of Definition \ref{def1}, we have
	\begin{align*}
		u\in C^{0,\gamma}_\mathrm{loc}(\Omega)
	\end{align*}
	for any $\gamma\in(0,1)$. Moreover, there exist constants $C$ and $\kappa$ depending on $N, p, s_1, s_p, \gamma$, such that for any ball $B_R\equiv B_R(x_0)\subset\subset\Omega$ with $R\in(0,1)$ and any $r\in(0,R)$,
	\begin{align*}
		[u]_{C^{0,\gamma}(B_r)}\le \frac{C\left( \mathcal{T}+[u]_{W^{s_p,p}(B_R)}+1\right) }{(R-r)^{\kappa}}.
	\end{align*}
\end{corollary}

\begin{remark}
	In our iterative scheme, we define the following sequence to track the evolution of weak differentiability:
	\begin{align*}
		\gamma_0=\gamma,\quad \gamma_{i+1}=\gamma_i\left(1-\frac{p^2}{2 q}+\frac{p}{2 q}\right)+\frac{s_pp^2}{2 q}.
	\end{align*}
	For the case $s_p \in \left(0,  \frac{p-1}{p} \right]$, we have
	\begin{align*}
		\lim\limits_{i\rightarrow+\infty}\gamma_i=\frac{s_pp}{p-1}.
	\end{align*}
	For the case $s_p \in \left( \frac{p-1}{p}, 1 \right)$, since $\gamma_i$ may exceed 1, the limitations of the difference-quotient technique prevent us from iterating the first- and second-order differences indefinitely.
\end{remark}

Provided that the gradient of weak solutions exists, one can also establish the Sobolev regularity of $Du$ as follows.

\begin{corollary}
	\label{corG}
	Let $q\ge2$, $s_p\in\left( \frac{p-1}{p},1\right) $ and let $u$ be a locally bounded weak solution to problem \eqref{1.1} in the sense of Definition \ref{def1}. Then, we have
	\begin{align*}
		u\in W^{1+\alpha,q}_\mathrm{loc}(\Omega)
	\end{align*}
	for any $\alpha\in \left( 0,\frac{p}{2q}(s_pp-p+1)\right) $. Moreover, there exist constants $C$ and $\kappa$ depending on $N, p, s_1, s_p,\alpha$, such that for any ball $B_R\equiv B_R(x_0)\subset\subset\Omega$ with $R\in (0,1)$ and any $r\in(0,R)$,
	\begin{align*}
		[\nabla u]^{q}_{W^{\alpha,q}(B_r)}\le \frac{C\left( \mathcal{T}+[u]_{W^{s_p,p}(B_R)}+1\right)^q}{(R-r)^{\kappa}}.
	\end{align*}
\end{corollary} 

This paper is organized as follows. In Section \ref{sec2}, we introduce the notations for fractional Sobolev spaces and present several fundamental lemmas. Section \ref{sec3} is devoted primarily to establishing energy estimates for weak solutions. Finally, in Section \ref{sec4}, we derive the Sobolev regularity of weak solutions using the finite-difference quotient method combined with a Moser-type iteration scheme.

\section{Preliminaries}
\label{sec2}
In this section, we introduce the notations and several auxiliary lemmas that will be employed in the proofs of main results.

\subsection{Notations}

For a vector $z\in\mathbb{R}^N$, we denote its Euclidean norm by $|z|$. Throughout the manuscript, the symbol $C$ stands for a positive constant whose value may vary from one occurrence to another, possibly even within the same line. Whenever relevant, we explicitly specify its dependence on parameters, for example, a constant depending on $N,p,s_1,s_p$ is denoted by $C(N,p,s_1,s_p)$.

For a function $w:\Omega\rightarrow\mathbb{R}$, if $w\in C^{0,\alpha}(B_R(x_0))$ for some $\alpha\in(0,1)$ and for every ball $B_R(x_0)\subset\subset\Omega$, we write $w\in C^{0,\alpha}_\mathrm{loc}(\Omega)$. The seminorm of this H\"{o}lder space is defined by
\begin{align*}
	[w]_{C^{0,\alpha}(B_R(x_0))}:=\sup\limits_{x,y\in B_R(x_0),x\neq y}\frac{|w(x)-w(y)|}{|x-y|^{\alpha}}.
\end{align*}

Next, we introduce the (fractional) Sobolev space. For $q\ge 1$ and $\gamma\in (0,1)$, the space $W^{1,q}(\Omega)$ and $W^{\gamma,q}(\Omega)$ are respectively defined as
\begin{align*}
	\left\lbrace u\in L^q(\Omega)\bigg|\|u\|_{W^{1,q}(\Omega)}:=\|u\|_{L^q(\Omega)}+\|\nabla u\|_{L^q(\Omega)}<+\infty\right\rbrace 
\end{align*}
and
\begin{align*}
	\left\lbrace u\in L^q(\Omega)\bigg|[u]_{W^{\gamma,q}(\Omega)}:=\left( \int_{\Omega}\int_{\Omega}\frac{|u(x)-u(y)|^q}{|x-y|^{N+\gamma q}}\,dxdy\right)^\frac{1}{q} <+\infty\right\rbrace .
\end{align*}
We define $\mathrm{Tail}(u;x_0,R)$ to indicate the nonlocal information of $u\in  L^p_{\mathrm{loc}}(\mathbb{R}^N)$ as
\begin{align*}
	\mathrm{Tail}(u;x_0,R):=\left( R^{s_pp}\int_{\mathbb{R}^N\backslash B_R(x_0)}\frac{|u|^{p-1}}{|x-x_0|^{N+s_p p}}\,dx\right)^\frac{1}{p-1}. 
\end{align*}

\subsection{Basic estimates and embedding results}
Suppose that $u$ is locally bounded in $\Omega$. We recall the following lemma, which was originally established in \cite[Lemma 2.3]{BLS18}.
\begin{lemma}
	\label{lem4}
	Let $p>1$ and $s_p\in (0, 1)$. For any $u\in L^{p-1}_{s_pp}(\mathbb{R}^N)$ and any ball $B_R\equiv B_R(x_0)\subset\subset\Omega$, $r\in (0,R)$, we have
	\begin{align*}
		\mathrm{Tail}(u;x_0,r)^{p-1}\le C(N)\left( \frac{R}{r}\right)^N\left( \mathrm{Tail}(u;x_0,R)+\|u\|_{L^\infty(B_R)}\right)^{p-1}. 
	\end{align*}
\end{lemma}

Additionally, we state the well-known embedding results for Sobolev functions. Lemmas \ref{lemMorrey1} and \ref{embed} are derived from \cite{NO23, BDL2401}, while Lemma \ref{Morrey2} can be found in \cite{L19}.

\begin{lemma}[\text{Embedding $W^{\gamma,q}\hookrightarrow C^{0,\gamma-\frac{N}{q}}$}]
	\label{lemMorrey1}
	Let $q\ge1$ and $\gamma\in (0, 1)$ such that $\gamma q >N$. Then, there exists a constant $C=C(N, q, \gamma)$ such that for any $w\in W^{\gamma,q}(B_R)$, we have
	\begin{align*}
		[w]_{C^{0,\gamma-\frac{N}{q}}(B_R)}\le C[w]_{W^{\gamma, q}(B_R)}.
	\end{align*}
\end{lemma}

\begin{lemma}[\text{Embedding $W^{1,q}\hookrightarrow W^{\gamma,q}$}]
	\label{embed}
	Let $q\ge1$ and $\gamma\in(0,1)$. Then, for any $w\in W^{1,q}(B_R)$, we have
	\begin{align}
		\label{embedd}
		\int_{B_R}\int_{B_R}\frac{|w(x)-w(y)|^q}{|x-y|^{N+\gamma q}}\,dxdy\le C(N)\frac{R^{(1-\gamma)q}}{(1-\gamma)q}\int_{B_R}|\nabla w|^q\,dx.
	\end{align}
\end{lemma}

\begin{lemma}[\text{Embedding $W^{1, q}\hookrightarrow C^{0,1-\frac{N}{q}}$}]
	\label{Morrey2}
	Let $q> N$. There exists a constant $C=C(N, q)$ such that for any $w\in W^{1,q}(B_R)$,
	\begin{align*}
		[w]_{C^{0,1-\frac{N}{q}}(B_R)}\le C\|\nabla w\|_{L^q(B_R)}.
	\end{align*}
\end{lemma}

\subsection{Algebraic inequalities.}
For $\gamma>0$ and $a\in\mathbb{R}$, we define
\begin{align*}
	J_\gamma(a):=|a|^{\gamma-2}a.
\end{align*}

The following two algebraic inequalities, taken from \cite{BDL2401} and \cite{BDL2402}, will be useful in deriving energy estimates for weak solutions.

\begin{lemma}
	\label{lem2}
	For any $\gamma>0$ and any $a,b\in \mathbb{R}$, we have
	\begin{align*}
		C_1\left( |a|+|b|\right) ^{\gamma-1}|b-a|\le \left| J_{\gamma+1}(b)-J_{\gamma+1}(a)\right| \le C_2\left( |a|+|b|\right) ^{\gamma-1}|b-a|,
	\end{align*}
	where $C_1=\min\left\lbrace \gamma,2^{1-\gamma}\right\rbrace $ and $C_2=\max\left\lbrace \gamma,2^{1-\gamma}\right\rbrace $.
\end{lemma}

\begin{lemma}
	\label{lem3}
	Let $p\in(1,2)$ and $\delta\ge1$, for any $a,b,c,d\in\mathbb{R}$ and any $e,f\in \mathbb{R}^+$, we have
	\begin{align*}
		&\quad\left( J_p(a-b)-J_p(c-d)\right)\left( J_{\delta+1}(a-c)e^p-J_{\delta+1}(b-d)f^p\right)\\
		&\ge \frac{p-1}{2^{\delta+1}}\left( |a-b|+|c-d|\right)^{p-2}\left( |a-c|+|b-d|\right)^{\delta-1}\left| (a-c)-(b-d)\right|^2\left( e^p+f^p\right)\\
		&\quad -\left( \frac{2^{\delta+1}}{p-1}\right)^{p-1}\left( |a-c|+|b-d|\right)^{p+\delta-1}|e-f|^p . 
	\end{align*}
\end{lemma}

Although the advantages of Lemma \ref{lem3} have not been fully exploited in this work, we still employ it in proving the main results so that the readers can more clearly grasp the influence of the $1$-structure on the problem.

\subsection{Finite difference technique}
For a direction vector $h\in \mathbb R^N$ and a measurable function $u\in \Omega\to \mathbb R$, we define the difference operator 
\begin{align*}
	\tau_hu(x):=u(x+h)-u(x), \quad x\in\mathbb{R}^N.
\end{align*}
To investigate the local behavior of weak solutions to problem \eqref{1.1}, we present several technical lemmas used in the application of the finite-difference method. These lemmas are summarized in \cite{GT83, BL17, DKLN23, D04,LZ25}.
\begin{lemma} 
	\label{lem5}
	Let $q>1$, $\gamma\in (0,1)$, $R>0$ and $d\in (0,R)$. Then, there exists a constant $C=C(N,q)$ such that for any $w\in W^{\gamma,q}(B_R)$,
	\begin{align*}
		\int_{B_{R-d}}|\tau_hw|^q\,dx\le C|h|^{\gamma q}\left[ (1-\gamma)[w]^q_{W^{\gamma,q}(B_R)}+\left( \frac{R^{(1-\gamma)q}}{d^q}+\frac{1}{\gamma d^{\gamma q}}\right)\|w\|^q_{L^q(B_R)} \right] 
	\end{align*}
	for any $h\in B_d\backslash \left\lbrace 0\right\rbrace$.
\end{lemma}

\begin{lemma}
	\label{lem6}
	Let $q\ge1$, $\gamma>0$, $M\ge0$, $0<r<R$ and $d\in \left( 0,\frac{1}{2}(R-r)\right] $. Then, there exists a constant $C=C(q)$ such that whenever $w\in L^q(B_R)$ satisfies
	\begin{align*}
		\int_{B_r}|\tau_h(\tau_hw)|^q\,dx\le M^q|h|^{\gamma q}
	\end{align*}
	for any $h\in B_d\backslash \left\lbrace 0\right\rbrace $. Then in the case $\gamma\in (0,1)$, we have, for any $h\in B_{\frac{1}{2}d}\backslash\left\lbrace 0\right\rbrace $,
	\begin{align*}
		\int_{B_r}|\tau_hw|^q\le C(q)|h|^{\gamma q}\left[ \left( \frac{M}{1-\gamma}\right)^q+\frac{1}{d^{\gamma q}}\int_{B_R}|w|^q\,dx \right].
	\end{align*}
	In the case $\gamma>1$, we have
	\begin{align*}
		\int_{B_r}|\tau_hw|^q\,dx\le C(q)|h|^q\left[ \left( \frac{M}{\gamma-1}\right)^qd^{(\gamma-1)q}+\frac{1}{d^q}\int_{B_R}|w|^q\,dx \right].
	\end{align*}
	In the case $\gamma=1$, we have
	\begin{align*}
		\int_{B_r}|\tau_hw|^q\,dx\le C(q)|h|^{\lambda q}\left[ \left( \frac{M}{1-\lambda}\right)^qd^{(1-\lambda)q}+\frac{1}{d^{\lambda q}}\int_{B_R}|w|^q\,dx \right] 
	\end{align*}
	for any $0<\lambda<1$.
\end{lemma}

\begin{lemma}
	\label{lem7}
	Let $q\ge1$, $\gamma\in (0,1]$, $M\ge1$ and $0<d<R$. Then, there exists a constant $C=C(N,q)$ such that whenever $w\in L^q(B_{R+d})$ satisfies
	\begin{align*}
		\int_{B_R}|\tau_hw|^q\,dx\le M^q|h|^{\gamma q}
	\end{align*}
	for any $h\in B_d\backslash\left\lbrace 0\right\rbrace $, $w\in W^{\beta,q}(B_R)$ for any $\beta\in (0,\gamma)$. Moreover, we have
	\begin{align*}
		[w]^q_{W^{\beta,q}(B_R)}\le C\left[ \frac{d^{(\gamma-\beta)q}}{\gamma-\beta}M^q+\frac{1}{\beta d^{\beta q}}\|w\|^q_{L^q(B_R)}\right]. 
	\end{align*}
\end{lemma}

\begin{lemma}
	\label{lem16}
	Let $q>1$, $M>0$ and $0<d<R$. Then, any $w\in L^q(B_R)$ that satisfies
	\begin{align*}
		\int_{B_{R-d}}|\tau_hw|^q\,dx\le M^q|h|^q
	\end{align*}
	for any $h\in B_d\backslash \left\lbrace 0\right\rbrace $ is weakly differentiable in $B_{R-d}$. Moreover, we have
	\begin{align*}
		\int_{B_{R-d}}|Du|^q\,dx\le C(N)M^q.
	\end{align*} 
\end{lemma}

\begin{lemma}
	\label{lem21}
	Let $q\ge1$, $\gamma\in(0, 1)$, $M>0$, $R>0$ and $d\in(0,R)$.  For any $w\in W^{1,q}(B_{R+6d})$ that satisfies
	\begin{align*}
		\int_{R+4d}|\tau_h(\tau_hw)|^q\,dx\le M^q|h|^{q(1+\gamma)}
	\end{align*}
	for any $h\in B_d\left\lbrace 0\right\rbrace $, we have
	\begin{align*}
		\nabla w\in \left( W^{\beta,q}(B_R)\right)^N 
	\end{align*}
	for any $\beta\in(0, \gamma)$. Moreover, there exists a constant $C=C(N, q)$ such that
	\begin{align*}
		[\nabla w]^q_{W^{\beta,q}(B_R)}\le \frac{Cd^{q(\gamma-\beta)}}{(\gamma-\beta)\gamma^q(1-\gamma)^q}\left(M^q+\frac{(R+4d)^q}{\beta d^{q(1+\gamma)}}\int_{B_{R+4d}}|\nabla w|^q\,dx\right). 
	\end{align*}
\end{lemma}

\section{Integral estimates}
\label{sec3}
In this section, we present two integral estimates that arise in the analysis of formulation \eqref{1.2}.
The estimate below, which originates from \cite[Lemma 3.2]{BDL2401}, is devoted to handling the nonlocal component of the $p$-growth term.

\begin{lemma}
	\label{lem8}
	Let $p\in(1, 2)$ and $s_p\in(0, 1)$. There exists a constant $C:=\frac{\tilde{C}(N,p)}{s_p}$ such that whenever $u\in L^{p-1}_{s_pp}(\mathbb{R}^N)$, $x_0\in\mathbb{R}^N$, $R>0$, $r\in(0,r)$ and $d\in \left( 0,\frac{1}{4}(R-r)\right] $, for any $x\in B_{\frac{1}{2}(R+r)}(x_0)$ and any $h\in B_d\backslash \left\lbrace 0\right\rbrace $, we have
	\begin{align*}
		&\quad\left| \int_{\mathbb{R}^N\backslash B_R(x_0)}\frac{J_p\left( u_h(x)-u_h(y)\right)-J_p\left( u(x)-u(y)\right) }{|x-y|^{N+s_pp}}\,dy\right|\\
		&\le \frac{C}{R^{s_pp}}\left( \frac{R}{R-r}\right)^{N+s_pp+1}\left( |\tau_hu(x)|^{p-1}+\frac{|h|}{R}\mathcal{T}^{p-1}\right),
	\end{align*}
	where
	\begin{align*}
		\mathcal{T}:=\|u\|_{L^\infty(B_{R+d}(x_0))}+\mathrm{Tail}(u;x_0,R+d).
	\end{align*}
\end{lemma}

\begin{remark}
	Compared with the result in \cite{BL17}, Lemma \ref{lem8} ensures the existence of the gradient $Du$ for the fractional $p$-Laplacian in the subquadratic case, regardless of the value of $s$. However, because the 1-growth term is involved in our problem and the function $Z$ is not expected to yield meaningful information on $h$, the advantage of Lemma \ref{lem8} is diminished.
\end{remark}

Next, we obtain an energy estimate for weak solutions, which shows that the $p$-growth term is controlled by the $1$-growth term.

\begin{proposition}
	\label{pro9}
	Let $p\in(1,2)$, $s_1,s_p,\gamma\in(0,1)$ and $q\ge p$. There exists a constant $C$ depending on $N,p,q,s_1,s_p,\sigma,\gamma$ such that for any $u\in W^{\gamma,q}_{\mathrm{loc}}(\Omega)$ being a locally bounded weak solution to problem \eqref{1.1} in the sense of Definition \ref{def1}, we have, for any $R\in(0,1)$, $r\in(0,R)$, $d\in\left(\frac{1}{4}(R-r) \right] $, $B_{R+d}\equiv B_{R+d}(x_0)\subset\subset\Omega$, $h\in \overline{B_d}\backslash\left\lbrace 0\right\rbrace $ and $\sigma\in(0,\gamma p)$,
	\begin{align*}
		&\quad\int_{B_r}\int_{B_r}\frac{\left| J_{\frac{q}{p}+1}\left( \tau_hu(x)\right)-J_{\frac{q}{p}+1}\left( \tau_hu(y)\right) \right|^p }{|x-y|^{N+\frac{s_pp^2}{2}+\sigma\left( 1-\frac{p}{2}\right) }}\,dxdy\\
		&\le \left( \frac{C}{(R-r)^{N+s_pp+1}}\right)^\frac{p}{2}\left[ [u]^p_{W^{\gamma,q}(B_r)}\left( \int_{B_r}|\tau_hu|^q\,dx\right)^\frac{q-p}{q} \right]^{1-\frac{p}{2}}\\
		&\quad\times\left(\int_{B_R}|\tau_hu|^{q-p+1}\,dx+\int_{B_R}|\tau_hu|^q\,dx+|h|\mathcal{T}^{p-1}\int_{B_R}|\tau_hu|^{q-p+1}\,dx\right) ^\frac{p}{2}
	\end{align*}
	with
	\begin{align*}
		\mathcal{T}:=\|u\|_{L^\infty(B_{R+d}(x_0))}+\mathrm{Tail}(u;x_0,R+d).
	\end{align*}
\end{proposition}

\begin{proof}
	For a ball $B_{R+d}\equiv B_{R+d}(x_0)\subset\subset\Omega$, $0<r<R\le1$ and $d\in\left(0,\frac{1}{4}(R-r) \right] $, by choosing $\varphi_{-h}(x):=\varphi(x-h)$ as a test function in \eqref{1.2} that satisfies $\varphi\in W^{s_1, 1}_{\mathrm{loc}}(B_R)\cap W^{s_p, p}_{\mathrm{loc}}(B_R)$ and $\mathrm{supp}\,\varphi\in B_{\frac{1}{2}(R+r)}$, we have
	\begin{align}
		\label{9.1}
		0&=\int_{\mathbb{R}^N}\int_{\mathbb{R}^N}Z_h\frac{\varphi(x)-\varphi(y)}{|x-y|^{N+s_1}}\,dxdy\nonumber\\
		&\quad+\int_{\mathbb{R}^N}\int_{\mathbb{R}^N}\frac{|u_h(x)-u_h(y)|^{p-2}(u_h(x)-u_h(y))(\varphi(x)-\varphi(y))}{|x-y|^{N+s_pp}}\,dxdy.
	\end{align}
	After subtracting \eqref{1.2} from \eqref{9.1}, one can obtain
	\begin{align}
		\label{9.2}
		0&=\int_{\mathbb{R}^N}\int_{\mathbb{R}^N}(Z_h-Z)\frac{\varphi(x)-\varphi(y)}{|x-y|^{N+s_1}}\,dxdy\nonumber\\
		&\quad+\int_{\mathbb{R}^N}\int_{\mathbb{R}^N}\frac{\left[ J_p(u_h(x)-u_h(y))-J_p(u(x)-u(y))\right] (\varphi(x)-\varphi(y))}{|x-y|^{N+s_pp}}\,dxdy.
	\end{align}
	We now set a function $\eta$ satisfying
	\begin{align}
		\label{eta}
		0\le\eta\in C^1_0\left( B_{\frac{1}{2}(R+r)}\right),\ |\nabla\eta|\le\frac{C}{R-r}\quad\text{and}\quad\eta\equiv1\ \text{in}\ B_r. 
	\end{align}
	Since $u$ is locally bounded in $\Omega$, one can verify that $J_{q-p+2}(\tau_hu)\eta^p\in W^{s_1,1}_{0}(B_R)\cap W^{s_p,p}_0(B_R)$. Thus, letting $\varphi:=J_{q-p+2}(\tau_hu)\eta^p$, \eqref{9.2} implies
	\begin{align}
		\label{9.3}
		0&=\int_{\mathbb{R}^N}\int_{\mathbb{R}^N}(Z_h-Z)\frac{J_{q-p+2}(\tau_hu)\eta^p(x)-J_{q-p+2}(\tau_hu)\eta^p(y)}{|x-y|^{N+s_1}}\,dxdy\nonumber\\
		&\quad+\int_{\mathbb{R}^N}\int_{\mathbb{R}^N}\frac{J_p(u_h(x)-u_h(y))-J_p(u(x)-u(y))}{|x-y|^{N+s_pp}}\nonumber\\
		&\qquad\qquad\qquad\times\left[ J_{q-p+2}(\tau_hu)\eta^p(x)-J_{q-p+2}(\tau_hu)\eta^p(y)\right]\,dxdy\nonumber\\
		&=:I+P.
	\end{align}
	
We will deal with the $1$-growth term $I$ and $p$-growth term $P$ one by one.
\smallskip	

\noindent	\textbf{Estimate of $I$.} From \cite[Proof of Proposition 4.2]{LZ25}, we have
	\begin{align}
		\label{9.4}
		I\ge I_{1,2}+2I_2
	\end{align}
	with
	\begin{align*}
		I_{1,2}:=\int_{B_R}\int_{B_R}(Z_h-Z)\frac{\eta^p(x)-\eta^p(y)}{|x-y|^{N+s_1}}\frac{J_{q-p+2}(\tau_hu)(x)+J_{q-p+2}(\tau_hu)(y)}{2}\,dxdy
	\end{align*}
	and
	\begin{align*}
		I_2:=\int_{B_{\frac{1}{2}(R+r)}}\int_{\mathbb{R}^N\backslash B_R}(Z_h-Z)\frac{J_{q-p+2}(\tau_hu)\eta^p(x)}{|x-y|^{N+s_1}}\,dxdy.
	\end{align*}
	Moreover, one can estimate $I_{1,2}$ and $I_2$ from above, that is,
	\begin{align*}
		|I_{1,2}|\le \frac{C(N,p)}{(1-s_1)(R-r)}\int_{B_R}|\tau_hu|^{q-p+1}\,dx
	\end{align*}
	and
	\begin{align*}
		|I_2|\le \frac{C(N)}{s_1(R-r)^{s_1}}\int_{B_R}|\tau_hu|^{q-p+1}\,dx.
	\end{align*}

\smallskip

	\noindent \textbf{Estimate of $P$.}	Note that
	\begin{align}
		\label{9.5}
		P=P_1+2P_2
	\end{align}
	with
	\begin{align*}
		&P_1:=\int_{B_R}\int_{B_R}\frac{J_p(u_h(x)-u_h(y))-J_p(u(x)-u(y))}{|x-y|^{N+s_pp}}\\
		&\qquad\qquad\qquad\times\left[ J_{q-p+2}(\tau_hu)\eta^p(x)-J_{q-p+2}(\tau_hu)\eta^p(y)\right]\,dxdy
	\end{align*}
	and
	\begin{align*}
		P_2:=\int_{B_{\frac{1}{2}(R+r)}}\int_{\mathbb{R}^N\backslash B_R}\frac{J_p(u_h(x)-u_h(y))-J_p(u(x)-u(y))}{|x-y|^{N+s_pp}}J_{q-p+2}(\tau_hu)\eta^p(x)\,dxdy.
	\end{align*}
	Applying Lemma \ref{lem3} to the local term $P_1$, we get
	\begin{align}
		\label{9.6}
		P_1\ge\frac{p-1}{2^{q-p+2}}P_{1,1}-\left( \frac{2^{q-p+2}}{p-1}\right)^{p-1}P_{1,2},
	\end{align}
	where we define
	\begin{align*}
		P_{1,1}&:=\int_{B_R}\int_{B_R}\frac{\left( |u_h(x)-u_h(y)|+|u(x)-u(y)|\right)^{p-2}\left( |\tau_hu(x)|+|\tau_hu(y)|\right)^{q-p} }{|x-y|^{N+s_pp}}\nonumber\\
		&\qquad\qquad\qquad\times|\tau_hu(x)-\tau_hu(y)|^2\left( \eta^p(x)+\eta^p(y)\right)\,dxdy
	\end{align*}
	and
	\begin{align*}
		P_{1,2}:=\int_{B_R}\int_{B_R}\frac{\left( |\tau_hu(x)|+|\tau_hu(y)|\right)^q|\eta(x)-\eta(y)|^p }{|x-y|^{N+s_pp}}\,dxdy.
	\end{align*}
	One can estimate $P_{1,2}$ from \eqref{eta} as follows:
	\begin{align}
		\label{9.7}
		|P_{1,2}|&\le \frac{C}{(R-r)^p}\int_{B_R}\int_{B_R}\frac{\left( |\tau_hu(x)|+|\tau_hu(x)|\right)^q }{|x-y|^{N+(s_p-1)p}}\,dxdy\nonumber\\
		&\le \frac{C}{(R-r)^p}\int^{2R}_0\frac{d\rho}{\rho^{(s_p-1)p+1}}\int_{B_R}|\tau_hu|^q\,dx\nonumber\\
		&\le \frac{C(N,p)}{(1-s_p)(R-r)^p}\int_{B_R}|\tau_hu|^q\,dx.
	\end{align}
	By Lemma \ref{lem2} and H\"{o}lder's inequality with exponents $\left( \frac{2}{p},\frac{2}{2-p}\right) $, we deduce
	\begin{align}
		\label{9.8}
		&\quad\int_{B_r}\int_{B_r}\frac{\left| J_{\frac{q}{p}+1}\left( \tau_hu(x)\right)-J_{\frac{q}{p}+1}\left( \tau_hu(y)\right) \right|^p }{|x-y|^{N+\frac{s_pp^2}{2}+\sigma\left( 1-\frac{p}{2}\right) }}\,dxdy\nonumber\\
		&\le C(p,q)\int_{B_r}\int_{B_r}\frac{\left( |\tau_hu(x)|+|\tau_hu(y)|\right)^{q-p}|\tau_hu(x)-\tau_hu(y)|^p }{|x-y|^{N+\frac{s_pp}{2}+\sigma\left( 1-\frac{p}{2}\right) }}\,dxdy\nonumber\\
		&\le C(p,q)P_{1,3}^{1-\frac{p}{2}}P_{1,1}^{\frac{p}{2}},
	\end{align}
	with
	\begin{align*}
		P_{1,3}:=\int_{B_r}\int_{B_r}\frac{\left( |u_h(x)-u_h(y)|+|u(x)-u(y)|\right)^p\left( |\tau_h(x)|+|\tau_h(y)|\right)^{q-p}}{|x-y|^{N+\sigma}}\,dxdy.
	\end{align*}
	For the purpose of deriving the norm $[u]_{W^{\gamma,q}}$, we use H\"{o}lder's inequality with exponents $\left( \frac{q}{p},\frac{q}{q-p}\right) $ to get
	\begin{align}
		\label{9.9}
		P_{1,3}&\le\left( \int_{B_r}\int_{B_r}\frac{\left( |u_h(x)-u_h(y)|+|u(x)-u(y)|\right)^q }{|x-y|^{N+\gamma q}}\,dxdy\right)^{\frac{p}{q}}\nonumber\\
		&\quad\times\left( \int_{B_r}\int_{B_r}\frac{\left( |\tau_hu(x)|+|\tau_hu(y)|^q\right) }{|x-y|^{N+(\sigma-\gamma p)\frac{q}{q-p}}}\,dxdy\right)^{\frac{q-p}{q}}\nonumber\\
		&\le C(p,q)[u]^p_{W^{\gamma,q}(B_r)}\left( \int_0^{2R}\frac{d\rho}{\rho^{(\sigma-\gamma p)\frac{q}{q-p}+1}}\int_{B_r}|\tau_hu(x)|^q\,dx\right)^{\frac{q-p}{q}}\nonumber\\
		&\le \frac{C(p,q)}{(\gamma p-\sigma)^{\frac{q-p}{q}}}[u]^p_{W^{\gamma,q}(B_r)}\left( \int_{B_r}|\tau_hu(x)|^q\,dx\right)^{\frac{q-p}{q}}.
	\end{align}
	
	Next, we pay our attention to the nonlocal term $P_2$. By Lemma \ref{lem8}, we obtain
	\begin{align}
		\label{9.10}
		|P_2|&\le \int_{B_{\frac{1}{2}(R+r)}}\left| \int_{\mathbb{R}^N\backslash B_{R}(x_0)}\frac{J_p(u_h(x)-u_h(y))-J_p(u(x)-u(y))}{|x-y|^{N+s_pp}}\,dy\right|\nonumber\\
		&\qquad\qquad\quad\times \left|J_{q-p+2}(\tau_hu(x))\right|\eta^p(x)\,dx\nonumber\\
		&\le\frac{C}{R^{s_pp}}\left( \frac{R}{R-r}\right) ^{N+s_pp+1}\int_{B_{\frac{1}{2}(R+r)}}\left( |\tau_hu|^{p-1}+\frac{|h|\mathcal{T}}{R}\right)|\tau_hu|^{q-p+1}\eta^p\,dx\nonumber\\ 
		&\le\frac{C}{(R-r)^{N+s_pp+1}}\left( \int_{B_{\frac{1}{2}(R+r)}}|\tau_hu|^q\,dx+|h|\mathcal{T}\int_{B_{\frac{1}{2}(R+r)}}|\tau_hu|^{q-p+1}\,dx\right). 
	\end{align}
	
	Finally, combining \eqref{9.3}--\eqref{9.10}, we get the desired inequality
	\begin{align*}
		&\quad\int_{B_r}\int_{B_r}\frac{\left| J_{\frac{q}{p}+1}\left( \tau_hu(x)\right)-J_{\frac{q}{p}+1}\left( \tau_hu(y)\right) \right|^p }{|x-y|^{N+\frac{s_pp^2}{2}+\sigma\left( 1-\frac{p}{2}\right) }}\,dxdy\\
		&\le C(p,q)P_{1,1}^{\frac{p}{2}}P_{1,3}^{1-\frac{p}{2}}\\
		&\le C(p,q)\left( I_{1,2}+I_2+P_2+P_{1,2}\right)^{\frac{p}{2}}P_{1,3}^{1-\frac{p}{2}}\\
		&\le \frac{C(p,q,N,s_1,s_p)}{(\gamma p-\sigma)^{\frac{q-p}{q}\left(1-\frac{p}{2}\right)}}\left([u]^{p}_{W^{\gamma,q}(B_r)}\int_{B_r}|\tau_hu|^q\,dx\right)^{\frac{q-p}{q}\left(1-\frac{p}{2}\right)}\\
		&\quad\times\left( \frac{\int_{B_R}|\tau_hu|^{q-p+1}\,dx}{R-r}+\frac{\int_{B_R}|\tau_hu|^{q-p+1}\,dx}{(R-r)^{s_1}}+\frac{\int_{B_{\frac{1}{2}(R+r)}}|\tau_hu|^{q}\,dx}{(R-r)^{N+s_pp+1}}\right.\\
		&\qquad\qquad\left.+\frac{\mathcal{T}^{p-1}|h|\int_{B_{\frac{1}{2}(R+r)}}|\tau_hu|^{q-p+1}\,dx}{(R-r)^{N+s_pp+1}}+\frac{\int_{B_R}|\tau_hu|^{q}\,dx}{(R-r)^p}\right)^{\frac{p}{2}}\\
		&\le \left( \frac{C}{(R-r)^{N+s_pp+1}}\right)^\frac{p}{2}\left[ [u]^p_{W^{\gamma,q}(B_r)}\left( \int_{B_r}|\tau_hu|^q\,dx\right)^\frac{q-p}{q} \right]^{1-\frac{p}{2}}\\
		&\quad\times\left(\int_{B_R}|\tau_hu|^{q-p+1}\,dx+\int_{B_R}|\tau_hu|^q\,dx+|h|\mathcal{T}^{p-1}\int_{B_R}|\tau_hu|^{q-p+1}\,dx\right) ^\frac{p}{2}.
	\end{align*}
This finishes the proof.
\end{proof}

\begin{remark}
	\label{rem10}
	In Proposition \ref{pro9}, the constant $C$ is given by $C:=\frac{\widetilde{C}(N,p,q,s_1,s_p)}{(\gamma p-\sigma)^\frac{q-p}{q}}$.
\end{remark}

\section{Sobolev regularity}
\label{sec4}

In this section, we will prove our main results by considering two distinct cases, i.e.,  
\begin{itemize}
	\item [(i)] For the case $s_p\in \left(0,  \frac{p-1}{p}\right]$, we show that the fractional differentiability index can be improved to any value less than $\frac{s_p p}{p-1}$. 
	
	\item [(ii)] For the case $s_p\in \left( \frac{p-1}{p}, 1\right) $, we prove that $Du$ is almost locally bounded.
\end{itemize}
Unlike the approach used in \cite{LZ25}, we devote only limited effort to raising the integrability from $p$ to any $q\ge p$. Since Proposition \ref{pro9} imposes no restriction on $\gamma$, it allows us to increase the integrability to any $q$. Although this procedure entails a considerable loss of differentiability, we can compensate for it by iteratively applying the improvements stated in Lemmas \ref{lem11} and \ref{lem140}.

\subsection{The case $s_p\in \left( 0,\frac{p-1}{p}\right]$}
We first demonstrate the improvement of the differentiability index through the following two lemmas.
\begin{lemma}
	\label{lem11}
	Let $p\in(1, 2)$, $s_p\in\left( 0, \frac{p-1}{p}\right]$, $q\ge p$ and $\gamma\in \left( 0, \frac{s_p p}{p-1}\right)$. Suppose that $u\in W^{\gamma,q}_\mathrm{loc}(\Omega)$ is a locally bounded weak solution to problem \eqref{1.1}. Then, we have
	\begin{align*}
		u\in W^{\alpha,q}_{\mathrm{loc}}(\Omega)
	\end{align*}
	for any $\alpha\in (\gamma,\beta)$, where $\beta:=\gamma\left( 1-\frac{p^2}{2q}+\frac{p}{2q}\right)+\frac{s_pp^2}{2q}$.
	Moreover, there exists a constant $C$ depending on $N,p,q,s_1,s_p,\alpha$ such that for any ball $B_R\equiv B_R(x_0)\subset\subset\Omega$ with $R\in(0,1)$ and any $r\in(0,R)$,
	\begin{align*}
		[u]^q_{W^{\alpha,q}(B_r)}\le \frac{C\left( \mathcal{T}+[u]_{W^{\gamma,q}(B_R)}+1\right)^q }{(R-r)^{N+2q+1}},
	\end{align*}
	where
	\begin{align*}
		\mathcal{T}=\|u\|_{L^\infty(B_{R}(x_0))}+\mathrm{Tail}(u;x_0,R).
	\end{align*}
\end{lemma}

\begin{proof}
	Note that $u\in W^{\gamma,q}_{\mathrm{loc}}(\Omega)$, we first apply Lemma \ref{lem5} to obtain 
	\begin{align}\label{11.0}
		\int_{B_{\widetilde{R}}}\left|\tau_h u\right|^q d x \leq C|h|^{\gamma q}\left[(1-\gamma)[u]_{W^{\gamma, q}\left(B_R\right)}^q+\left(\frac{R^{(1-\gamma) q}}{d^q}+\frac{1}{\gamma d^{\gamma q}}\right)\|u\|_{L^{\infty}(B_R)}^q\right]
	\end{align}
	for any $h \in \overline{{B}_d} \setminus\{0\}$, where $d:=\frac{1}{7}(R-r)$ and $\widetilde{R}:=R-d$.\par
	Thus, setting $\tilde{r}=r+d$ and from Proposition \ref{pro9}, one can deduce the following estimate 
	\begin{align}
		\label{11.1}
		&\quad\int_{B_{\tilde{r}}} \int_{B_{\tilde{r}}} \frac{\left| J_{\frac{q}{p}+1}\left(\tau_h u(x)\right)-J_{\frac{q}{p}+1}\left(\tau_h u(y)\right)\right|^p}{|x-y|^{N+s_p p  \frac{p}{2}+\sigma\left(1-\frac{p}{2}\right)}} d x d y \nonumber \\
		&\leq \frac{C[u]_{W^{\gamma,q}\left(B_r\right)}^{p\left(1-\frac{p}{2}\right)}}{(R-r)^{(N+s_p p+1)\frac{p}{2}}}\left[\left([u]_{W^{\gamma,q}\left(B_R\right)}^q+\frac{\| u \|_{L^{\infty}\left(B_R\right)}^q}{(R-r)^q}\right)|h|^{\gamma q}\right]^{\frac{q-p}{q}\left(1-\frac{p}{2}\right)} \nonumber\\
		&\quad \times\left[\left(\int_{B_{\widetilde{R}}}\left|\tau_h u\right|^q d x\right)^{\frac{q-p+1}{q}}+\int_{B_{\widetilde{R}}}\left|\tau_h u\right|^q d x+|h| \mathcal{T}^{p-1}\left(\int_{B_{\widetilde{R}}}\left|\tau_h u\right|^q d x\right)^{\frac{q-p+1}{q}}\right]^{\frac{p}{2}}\nonumber\\
		&\leq \frac{C[u]^{p\left(1-\frac{p}{2}\right)}}{(R-r)^{(N+s_p p+1)\frac{p}{2}}}\left[\left([u]_{W^{\gamma, q}\left(B_R\right)}^q+\frac{\|u\|_{L^\infty\left(B_R\right)}^q}{(R-r)^q}\right)|h|^{\gamma q}\right]^{\frac{q-p}{q}\left(1-\frac{p}{2}\right)}\nonumber \\
		& \quad\times\Bigg[\left(\left([u]_{W^{\gamma, q}\left(B_R\right)}^q+\frac{\|u\|_{L^{\infty}\left(B_R\right)}^q}{(R-r)^q}\right)|h|^{\gamma q}\right)^{\frac{q-p+1}{q}}+\left([u]_{w^{\gamma ,q}\left(B_R\right)}^q+\frac{\|u\|_{L^\infty\left(B_R\right)}^q}{(R-r)^q}\right)|h|^{\gamma q}\nonumber\\
		&\quad\qquad+|h| \mathcal{T}^{p-1}\left(\left([u]_{W^{\gamma,q}\left(B_R\right)}+\frac{\|u\| _{L^\infty\left(B_R\right)}}{(R-r)^q}\right)|h|^{\gamma q}\right)^{\frac{q-p+1}{q}}\Bigg]^{\frac{p}{2}}\nonumber\\
		&\leq  \frac{C[u]_{W^{\gamma,q}\left(B_R\right)}^{p\left(1-\frac{p}{2}\right)}}{(R-r)^{N+2 q+1}}\left([ u]_{W^{\gamma, q}\left(B_R\right)}+\|u\|_{L^{\infty}\left(B_R\right)}\right)^{(q-p)\left(1-\frac{p}{2}\right)}|h|^{\gamma (q-p)\left(1-\frac{p}{2}\right)} \nonumber\\
		& \quad\times\Bigg[\left([u]_{W^{\gamma, q}\left(B_R\right)}+\|u\|_{L^{\infty}\left(B_R\right)}+1\right)^q|h|^{\gamma(q-p+1)}+\left([u]_{W^{\gamma, q}\left(B_R\right)}+\| u\|_{L^{\infty}\left(B_R\right)}\right)^q|h|^{\gamma q}\nonumber \\
		& \quad\qquad+\left(\mathcal{T}^q+\left([u]_{W^{\gamma, q}\left(B_R\right)}+\|u\|_{L^{\infty}\left(B_R\right)}\right)^q\right)|h|^{\gamma(q-p+1)+1}\Bigg]^{\frac{p}{2}}\nonumber\\
		&\leq \frac{C\left(\mathcal{T}+[u]_{w^{\gamma, q}\left(B_R\right)}+1\right)^q}{(R-r)^{N+2 q+1}}|h|^{\gamma \left(q-\frac{p}{2}\right)},
	\end{align}
	where Hölder inequality and Minkowski's inequality with exponents $\left( \frac{q}{q-p+1},\frac{q}{p-1}\right) $ are utilized  in the  calculation above.
	
	Now, let us deal with the term in the first line of \eqref{11.1}. By Lemma \ref{lem5}, we have, for any $\lambda \in \overline{{B}_d} \setminus\{0\}$,
	\begin{align}\label{11.2}
		& \int_{B_r}\left|\tau_\lambda\left(J_{\frac{q}{p}+1}\left(\tau_h u\right)\right)\right|^p d x \nonumber\\
		\leq & C|\lambda|^{s_p p \frac{p}{2}+\sigma\left(1-\frac{p}{2}\right)}\left(\left[J_{\frac{q}{p}+1}\left(\tau_h u\right)\right]^p_{W^{\frac{s_p p}{2}+\frac{\sigma}{p}\left(1-\frac{p}{2}\right), p}\left(B_{\tilde{r}}\right)}+\frac{\int_{B_{\tilde{r}}}\left|\tau_h u\right|^q dx}{s_p(R-r)^p}\right) .
	\end{align}
	Bringing \eqref{11.0} and \eqref{11.2} into \eqref{11.1}, we get 
	\begin{align}
		\label{11.3}
	&\quad	\int_{B_r}\left|\tau_\lambda\left(J_{\frac{q}{p}+1}\left(\tau_h u\right)\right)\right|^p d x \nonumber\\
		&\leq \frac{C\left(\mathcal{T}+[u]_{W^{\gamma,q }\left(B_R\right)}+1\right)^q}{(R-r)^{N+2 q+1}}\left(|h|^{\gamma\left(q-\frac{p}{2}\right)}+|h|^{\gamma q}\right)|\lambda|^{ \frac{s_pp^2}{2}+\sigma\left(1-\frac{p}{2}\right)}.
	\end{align}
	Choosing $\lambda=h$ and observing that 
	\begin{align*}
		C(p, q)\left|\tau_h\left(J_{\frac{q}{p}+1}\left(\tau_h u\right)\right)\right| \geq\left|\tau_h\left(\tau_h u\right)\right|^{\frac{q}{p}},
	\end{align*}
	one can estimate the left-hand side of \eqref{11.3} from below to derive 
	\begin{align}\label{11.30}
		\int_{B_r}\left|\tau_h\left(\tau_h u\right)\right|^q d x \leq \frac{C\left(\mathcal{T}+[u]_{W^{\gamma, q}\left(B_R\right)}+1\right)^q}{(R-r)^{N+2q+1} }|h|^{\gamma\left(q-\frac{p}{2}\right)+\sigma\left(1-\frac{p}{2}\right)+\frac{s_p p^2}{2}}
	\end{align}
	for any  $h \in \overline{{B}_d} \setminus\{0\}$. And a direct application of Lemma \ref{lem6} enables us to deduce
	\begin{align}\label{11.4}
		\int_{B_r}\left|\tau_h u\right|^q d x \leq C|h|^{\gamma\left(q-\frac{p}{2}\right)+\sigma\left(1-\frac{p}{2}\right)+\frac{s_p p^2}{2}}\left[\frac{\left(\mathcal{T}+[u]_{W^{\gamma, q}\left(B_R\right)}+1\right)^q}{(R-r)^{N+2 q+1} }+\frac{\|u\|_{L^\infty\left(B_R\right)}^q}{(R-r)^{3 p}}\right].
	\end{align}
	
	To ensure $u \in W^{\alpha, q}\left(B_r\right)$, for $\alpha\in(\gamma,\beta)$, we set 
	\begin{align*}
		\sigma=\frac{(\alpha+\beta)q-s_pp^2-\gamma(2q-p)}{2-p}.
	\end{align*}
	In this setting, we have 
	\begin{align*}
		\tilde{\alpha}:=\frac{\gamma\left(q-\frac{p}{2}\right)+\sigma\left(1-\frac{p}{2}\right)+\frac{s_p p^2}{2}}{q}>\alpha.
	\end{align*}
	After calculation, \eqref{11.4} implies that
	\begin{align*}
		\int_{B r}\left|\tau_h u\right|^q d x \leq \frac{C\left(\mathcal{T}+[u]_{W^{\gamma,q}\left(B_R\right)}+1\right)^q}{(1-\tilde{\alpha})^q(R-r)^{N+2 q+1} }|h|^{\tilde{\alpha} q}.
	\end{align*}
	As a result of applying Lemma \ref{lem7}, we arrive at 
	\begin{align}
		\label{11.5}
		\int_{B_r} \int_{B_r} \frac{|u(x)-u(y)|^q}{|x-y|^{N+\alpha q}} d x d y & \leq \frac{C}{(1-\tilde{\alpha})^q}\left[\frac{\left(\mathcal{T}+[u]_{w^{\gamma, q}\left(B_R\right)}+1\right)^q}{(\tilde{\alpha}-\alpha)(R-r)^{N+2 q+1} }+\frac{\|u\|_{L^{\infty}\left(B_R\right)}^q}{\alpha(R-r)^{\tilde{\alpha} q}}\right] \nonumber\\
		& \leq \frac{C\left(\mathcal{T}+[u]_{W^{\gamma, q}\left(B_R\right)}+1\right)^q}{(R-r)^{N+2q+1} },
	\end{align}
	where the constant $C$ takes the form of 
	\begin{align*}
		C=\frac{\widetilde{C}\left(N, p, q, s_1, s_p\right)}{(1-\alpha)^{q+1}  \alpha(\beta-\alpha)^{\frac{q-p}{p}}}.
	\end{align*}
	We remark that the term $(\beta-\alpha)^{\frac{q-p}{p}}$ arises in the application of Proposition \ref{pro9}.
	One can compute the quantity $\gamma p-\sigma$ to obtain
	\begin{align*}
		\gamma p-\sigma&=\gamma p-\frac{(\alpha+\beta)q-s_pp^2-\gamma(2q-p)}{2-p}\\
		&=\frac{\gamma p-\gamma p^2-(\alpha+\beta)q+s_pp^2+2\gamma q}{2-p}\\
		&\ge \frac{\gamma p-\gamma p^2-2\beta q+s_pp^2+2\gamma q+(\beta-\alpha)q}{2-p}\\
		&=\frac{(\beta-\alpha)q}{2-p}.
	\end{align*}
\end{proof}

We now iterate Lemma \ref{lem11} to achieve almost $W^{\frac{s_pp}{p-1},q}$-regularity of weak solutions.
\begin{lemma}\label{lem12}
	Let $p\in(1,2)$, $s_p\in \left(0,\frac{p-1}{p}\right]$, $q\geq p$ and $\gamma\in \left( 0, \frac{s_pp}{p-1}\right) $. Suppose that $u \in W_{\rm loc}^{\gamma, q}(\Omega)$ is a 
	locally bounded weak solution to problem \eqref{1.1}. Then, we have 	
	\begin{align*}
		u \in W_{\mathrm{loc}}^{\sigma, q}(\Omega)
	\end{align*}
	for any $\sigma \in\left(\gamma, \frac{s_p p}{p-1}\right)$. Moreover, there exist constants $\kappa$ and $C$ depending on $N,p,q,s_1,s_p,\gamma,\sigma$ such that for any ball $B_R\equiv B_R(x_0)\subset\subset\Omega$ with $R\in(0,1)$ and any $r\in(0,R)$,
	\begin{align*}
		[u]_{W^{\sigma, q}\left(B_r\right)}^q \leq \frac{C\left(\mathcal{T}+[u]_{W^{\gamma, q}\left(B_R\right)}+1\right)^q}{(R-r)^{\kappa}}.
	\end{align*}
\end{lemma}

\begin{proof}
	Setting $\tilde{\sigma}=\frac{1}{2}\left(\sigma+\frac{s_p p}{p-1}\right)$ and defining
	\begin{align*}
		\gamma_0=\gamma,\quad \gamma_{i+1}=\gamma_i\left(1-\frac{p^2}{2 q}+\frac{p}{2 q}\right)+\tilde{\sigma}  \frac{p^2-p}{2 q}
	\end{align*}
	for $i\in\mathbb{N}$. So we have $\gamma_i\rightarrow\tilde{\sigma}$ as $i\rightarrow+\infty$. Additionally, we set
	\begin{align*}
		\rho_i:=r+\frac{1}{2^i}(R-r), \quad \beta_i:=\gamma_i\left(1-\frac{p^2}{2 q}+\frac{p}{2 q}\right)+\frac{s_p p^2}{2 q}
	\end{align*}
	and 
	\begin{align*}
		\mathcal{T}_i:=\|u\|_{L^{\infty}\left(B_{\rho_i}\right)}+\mathrm{Tail} \left(u ; \rho_i\right) .
	\end{align*}
	As a result of applying Lemma \ref{lem4},
	\begin{align}\label{12.1}
		\mathrm { Tail }\left(u ; \rho_{i-1}\right)^{p-1} \leq C(N)\left(\frac{R}{\rho_{i-1}}\right)^N \mathcal{T}^{p-1} \leq C(N)  2^{i N} \mathcal{T}^{p-1}.
	\end{align}

	Next, we will apply Lemma \ref{lem11} with parameters $(\alpha, \beta, \sigma, r, R)$ replaced by $\left(\gamma_i, \beta_i, \gamma_{i-1}, \rho_i, \rho_{i-1}\right)$ to deduce
	\begin{align*}
		[u]_{W^{\gamma_i ,q}\left(B_{p_i}\right)}^q \leq \frac{C\left(\mathcal{T}_{i-1}+[u]_{W^{\sigma_{i-1}, q}\left(B_{\rho_{i-1}}\right)}+1\right)^q}{\left(\rho_{i-1}-\rho_i\right)^{N+2 q+1}} .
	\end{align*}
	The above inequality holds when $[u]_{W^{\gamma_ {i+1}, q}{\left(B_{\rho_{i-1}}\right)}}<+\infty$.
	By definition of $\{\rho_i\}$ and \eqref{12.1}, we have 
	\begin{align}\label{12.2}
		[u]_{W^{\gamma_i,q }\left(B_{\rho_i}\right)} ^q &\leq \frac{2^{i(N+2q+1)}C\left(C(N, p) 2^{\frac{i N}{p-1}} \mathcal{T}+[u]_{W^{\gamma_{i-1}, q}\left(B_{\rho_{i-1}}\right)}+1\right)^q}{(R-r)^{N+2q+1}} \nonumber\\
		&\leq \frac{C_i 2^{i(N+2q+1)+i\frac{Nq}{p-1}}\left(\mathcal{T}+[u] _{W^{\gamma_{i-1},q}\left(B_{\rho_{i-1}}\right)}+1\right)^q}{(R-r)^{N+2q+1}} .
	\end{align}
	One can iterate \eqref{12.2} to obtain 
	\begin{align}\label{12.3}
		[u]^q_{w^{\gamma_i, q}\left(B_{\rho_i}\right)} \leq \frac{3^{i}\left(\prod\limits_{j=1}^i C_j\right)  2^{\left(N+2q+1+\frac{Nq}{p-1}\right)\left(\sum\limits_{j=1}^i j\right)}\left(\mathcal{T}+[u]_{W^{\gamma_0, q}\left(B_R\right)}+1\right)^q}{(R-r)^{(N+2q+1)i}}
	\end{align}
	for any $i\in \mathbb{N}$. Suppose $i_0$ be the smallest integer such that $\gamma_{i_0} \geq \sigma$, that is,
	\begin{align*}
		i_0:=\left[\frac{\ln \frac{\tilde{\sigma}-\sigma}{\sigma-\gamma}}{\ln \left(1-\frac{p^2}{2 q}+\frac{p}{2 q}\right)}\right]+2 \leq \frac{\ln \frac{\frac{s_p p}{p-1}-\sigma}{2(\sigma-\gamma)}}{\ln \left(1-\frac{p^2}{2 q}+\frac{p}{2 q}\right)}+2 .
	\end{align*}
	For the seminorm $[u]_{W^{\sigma, q}\left(B_r\right)}$, \eqref{12.3} implies 
	\begin{align*}
		[u]_{W^{\sigma, q}\left(B_r\right)}^q &\leq 2^q[u]_{W^{\sigma_{i_0}, q}\left(B_{\rho_{i_0}}\right)}^q \\
		& \leq \frac{\left(\prod\limits_{i=1}^{i_0} C_j\right) 3^{i_0} 2^{\left(N+2q+1+\frac{Nq}{p-1}\right)\left(\sum\limits_{j=1}^{i_0} j\right)}\left(\mathcal{T}+[u]_{W^{\gamma, q}\left(B_R\right)}+1\right)^q}{(R-r)^{(N+2q+1) i_0}} .
	\end{align*}
	Since $C_j$ takes the form of 
	\begin{align*}
		C_j&=\frac{C\left(N, p, q, s_1, s_p\right)}{(1-\gamma_i)^{q+1}\gamma_i(\beta_i-\gamma_i)^{\frac{q-p}{q}}}\\
		& \le \frac{C\left(N, p, q, s_1, s_p\right)}{\gamma \left(1-\sigma\right)^{q+1}[s_pp^2-(p-1)\tilde{\sigma}]^{\frac{q-p}{q}}}=:C_*,
	\end{align*}
	and $C_*$ is independent of $j$, we have 
	\begin{align*}
		[u]_{W^{\gamma, q}\left(B_r\right)}^q \leq \frac{C_*^{i_0}  3^{i_0} 2^{\left(N+2q+1+\frac{Nq}{p-1}\right)\left(\sum\limits_{j=1}^{i_0} j\right)}\left(\mathcal{T}+[u]_{W^{\gamma,q}\left(B_R\right)}+1\right)^q}{(R-r)^{(N+2q+1)i_0}}.
	\end{align*}
	By setting $C:=C_*^{i_0}  3^{i_0}  2^{\left(N+2q+1+\frac{Nq}{p-1}\right)\left(\sum\limits_{j=1}^{i_0} j\right)}$ and $\kappa:=(N+2q+1) i_0$, we conclude the desired claim.
\end{proof}

In the remainder of this subsection, we raise the integrability order and apply Lemma \ref{lemMorrey1} to establish the H\"{o}lder regularity of weak solutions.

\begin{proof}[Proof of Theorem \ref{th12}]
	Since $u$ is locally bounded in $\Omega$, we can prove $u\in W^{\frac{s_pp}{q},q}_{\mathrm{loc}}(\Omega)$. More precisely, we have 
	\begin{align}\label{13.1}
		[u]_{W^{\frac{s_p p}{q}, q}\left(B_R\right)}^q  & = \int_{B_R} \int_{B_R} \frac{|u(x)-u(y)|^{q-p}|u(x)-u(y)|^p}{|x-y|^{N+s_p p}} d x d y \nonumber\\
		& \leq 2^{q-p}\|u\|_{L^{\infty}\left(B_R\right)}^{q-p}[u]_{W^{s_p, p}\left(B_R\right)}^p \nonumber\\
		& \leq \overline{C}\left( [u]_{W^{s_p, p}\left(B_R\right)}+\|u\|_{L^{\infty}\left(B_R\right)}\right)^q .
	\end{align}
	By Lemma \ref{lem12}, we replace $\gamma$ by $\frac{s_pp}{q}$ to get 
	\begin{align*}
		[u]_{W^{\sigma, q}\left(B_r\right)}^q &\le \frac{\widetilde{C}\left(\mathcal{T}+[u]_{W^{\frac{s _pp}{q},q}\left(B_R\right)}+1\right)^q}{(R-r)^{\kappa} } \\
		& \le \frac{\widetilde{C}\left(\mathcal{T}+\overline{C}^{\frac{1}{q}}\left([u]_{W^{s_p,p}\left(B_R\right)}+\|u\|_{L^{\infty}\left(B_R\right)}\right)+1\right)^q}{(R-r)^{\kappa}} \\
		& \le \frac{4^q \widetilde{C} \overline{C}\left(\mathcal{T}+[u]_{W^{s_p,p}\left(B_R\right)}+1\right)^q}{(R-r)^{\kappa} } .
	\end{align*}
	One can set $C:=4^q \widetilde{C} \overline{C}$ to deduce desired inequality.
\end{proof}

\begin{proof}[Proof of Corollary \ref{cor13}]
	Let $\tilde{\gamma}=\frac{1}{2}\left( \gamma+\frac{s_pp}{p-1}\right) $ and $q=\frac{2N}{\frac{s_pp}{p-1}-\gamma}$. We apply the Morrey type embedding Lemma \ref{lemMorrey1} with $\gamma$ replaced by $\tilde{\gamma}$ and the claim of Theorem \ref{th12} to obtain
	\begin{align*}
		[u]_{C^{0,\gamma}(B_r)}=[u]_{C^{0,\tilde{\gamma}-\frac{N}{q}}(B_r)}&\le C[u]_{W^{\tilde{\gamma},q}(B_r)}\le \frac{C\left( \mathcal{T}+[u]_{W^{s_p,p}(B_R)}+1\right) }{(R-r)^{\kappa}}
	\end{align*}
	for some $C$ and $\kappa$ depending on $N, p, q, s_1, s_p, \gamma$.
\end{proof}

\subsection{The case $s_p\in \left( \frac{p-1}{p},1\right) $}

We divide our proof into two steps: (i) from $W^{s_p,p}_\mathrm{loc}(\Omega)$ to $W^{\frac{s_pp}{q},q}_\mathrm{loc}(\Omega)$; (ii) from  $W^{\frac{s_pp}{q},q}_\mathrm{loc}(\Omega)$ to $W^{1,q}_\mathrm{loc}(\Omega)$ for any $q\ge p$. As in the previous subsection, the first step is straightforward. However, in the second step, one must carefully handle the parameter that appears in the final stage of the iteration;  see \eqref{15.1} for details.

\begin{lemma}
	\label{lem140}
	Let $p\in (1,2)$, $s_p\in \left( \frac{p-1}{p}, 1\right) $ and $q\ge p$. Suppose that $u\in W^{\gamma,q}_{\mathrm{loc}}(\Omega)$ is a locally bounded weak solution to problem \eqref{1.1}. Then, we have
	\begin{align*}
		u\in W^{\alpha,q}_{\mathrm{loc}}(\Omega)
	\end{align*} 
	for any $\alpha\in (\gamma,\min\left\lbrace \beta,1\right\rbrace )$ with $\beta:=\gamma\left(1-\frac{p^2}{2q}+\frac{p}{2q}\right)+\frac{s_pp^2}{2q}$. Moreover, there exists a constant $C$ depending on $N,p,q,s_1,s_p,\alpha$ such that for any ball $B_R\equiv B_R(x_0)\subset\subset\Omega$ with $R\in(0,1)$ and any $r\in (0,R)$,
	\begin{align*}
		[u]^q_{W^{\alpha,q}(B_r)}\le \frac{C\left( \mathcal{T}+[u]_{W^{\gamma,q}(B_R)}+1\right)^q }{(R-r)^{N+s_pp+1}}.
	\end{align*}
\end{lemma}

\begin{proof}
	We will prove the desired result by using the method developed in the proof of Lemma \ref{lem11}. For the case $\beta\le1$, since the power of $|h|$ that in the right-hand side of \eqref{11.30} is less than $q$, that is,
	\begin{align*}
		\gamma\left( q-\frac{p}{2}\right)+\sigma\left( 1-\frac{p}{2}\right)+\frac{s_pp^2}{2}<\gamma\left( q-\frac{p}{2}\right)+\gamma p\left( 1-\frac{p}{2}\right)+\frac{s_pp^2}{2}=\beta q\le q,
	\end{align*}
	one can directly obtain the inequality \eqref{11.5}.
	
	For the case $\beta>1$, we set
	\begin{align*}
		\sigma:=\frac{2-s_pp^2-\gamma(2q-p)}{2-p}\in (0,\gamma p)
	\end{align*}
	and estimate the term in the right-hand side of \eqref{11.30} from above to deduce
	\begin{align*}
		\int_{B_r}|\tau_h(\tau_hu)|^q\,dx\le \frac{C\left( \mathcal{T}+[u]_{W^{\gamma,q}(B_R)}+1\right)^q }{(R-r)^{N+2q+1}}|h|^q.
	\end{align*}
	By Lemma \eqref{lem6}, we have
	\begin{align*}
		\int_{B_r}|\tau_hu|^q\,dx&\le C(q)|h|^{\tilde{\alpha}q}\left[ \frac{C\left( \mathcal{T}+[u]_{W^{\gamma,q}(B_R)}+1\right)^q }{(1-\tilde{\alpha})^q(R-r)^{N+2q+1}}+\frac{R^N\|u\|^q_{L^\infty(B_R)}}{(R-r)^{\alpha q}}\right]\\
		&\le \frac{\widetilde{C}(N,p,q,s_1,s_p)\left( \mathcal{T}+[u]_{W^{\gamma,q}(B_R)}+1\right)^q}{(1-\tilde{\alpha})^q(R-r)^{N+2q+1}}|h|^{\tilde{\alpha}q},
	\end{align*}
	where $\tilde{\alpha}:=\frac{1+\alpha}{2}$.
	
	Finally, we utilize Lemma \ref{lem7} to obtain
	\begin{align*}
		[u]^q_{W^{\alpha,q}(B_r)}&\le C(N,q)\left[ \frac{C(N,p,q,s_1,s_p)\left( \mathcal{T}+[u]_{W^{\gamma,q}(B_R)}+1\right)^q}{(\tilde{\alpha}-\alpha)(1-\tilde{\alpha})^q(R-r)^{N+2q+1}}+\frac{R^N\|u\|^q_{L^\infty(B_R)}}{\alpha(R-r)^{\alpha q}}\right] \\
		&\le \frac{C\left( \mathcal{T}+[u]_{W^{\gamma,q}(B_R)}+1\right)^q}{(R-r)^{N+2q+1}},
	\end{align*}
	where the constant $C$ is given by
	\begin{align*}
		C:=\frac{\widetilde{C}(N,p,q,s_1,s_p)}{(1-\alpha)^{q+1}\alpha(\min\left\lbrace \beta,1\right\rbrace-\alpha )^{\frac{q-p}{q}}}.
	\end{align*}
\end{proof}

\begin{lemma}\label{lem14}
	Let $p\in (1,2)$, $s_p\in \left(\frac{p-1}{p},1\right)$, $\gamma\in (0,1)$ and $q\geq p$. Suppose that $u \in W_{\mathrm{loc}}^{\gamma, q}(\Omega)$ is
	a locally bounded weak solution to problem \eqref{1.1}. Then, we have
	\begin{align*}
		u \in W_{\mathrm{loc}}^{\alpha, q}(\Omega)
	\end{align*}
	for any $\alpha\in (\gamma,1)$. Moreover, there exist constants $C$ and $\kappa$ depending on $N, p, q, s_1, s_p, \alpha, \gamma$ such that 
	for any ball $B_R \equiv B_R\left(x_0\right)$ with $R\in(0,1)$ and any $r\in (0,R)$,
	\begin{align*}
		[u]_{W^{\alpha, q}\left(B_r\right)}^q \leq \frac{C\left(\mathcal{T}+[u]_{W^{\gamma, q}\left(B_R\right)}+1\right)^q}{(R-r)^{\kappa}}.
	\end{align*}
\end{lemma}

\begin{proof}
	Similar to the proof of Lemma \ref{lem11}, we have 
	\begin{align}\label{14.1}
		[u]_{W^{\sigma, q}\left(B_{\tilde{r}}\right)}^q \leq \frac{C\left(\mathcal{T}+ [u]_{W^{\gamma, q}\left(B_{\widetilde{R}}\right)}+1\right)^q}{\left(\widetilde{R}-\tilde{r}\right)^{N+2q+1}}.
	\end{align}
	for any ball $B_{\widetilde{R}} \equiv B_{\widetilde{R}}\left(x_0\right)$ with $\widetilde{R}\in(0,1)$ and any $\tilde{r}\in (0,\widetilde{R})$, 
	where $\sigma\in \left(\gamma, \text{min}\{1, \beta\}\right)$ and $\beta:=\gamma\left(1-\frac{p^2}{2 q}+\frac{p}{2 q}\right)+\frac{s_p p^2}{2 q}$.\par
	Next, we define 
	\begin{align*}
		\gamma_0=\gamma, \quad \gamma_{i+1}=\gamma_i\left(1-\frac{p^2}{2 q}+\frac{p}{2 q}\right)+\frac{(p^2-p)(1+\alpha)}{4 q}
	\end{align*}
	and appply (\ref{14.1}) with parameters $\alpha$, $\beta$, $\sigma$, $r$, $R$ replaced by $\gamma_i$, $\gamma_i\left(1-\frac{p^2}{2 q}+\frac{p}{2 q}\right)+\frac{p^2-p}{2 q}$, $\gamma_{i+1}$, $\rho_{i+1}$, $\rho_i$ respectively to deduce
	\begin{align}\label{14.2}
		[u]_{W^{\gamma_{i+1},q}\left(B_{\rho_{i+1}}\right)}^q \leq \frac{C\left(\mathcal{T}_i+[u]_{W^{\gamma_i,q}\left(B_R\right)}+1\right)^q}{\left(\rho_i-\rho_{i+1}\right)^{N+2q+1}},
	\end{align}
	where $\{\rho_i\}$ and $\mathcal{T}_i$ are defined as in the proof of Lemma \ref{lem12}.\par
	Thus, from (\ref{12.1}) we get 
	\begin{align*}
		[u]_{W^{\gamma_i,q}\left(B_{\rho_ i}\right)}^q \leq \frac{C_i  2^{i\left(N+2q+1+\frac{Nq}{p-1}\right)}\left(\mathcal{T}+[u]_{W^{\gamma _{i-1}, q\left(B_{\rho_{i-1}}\right)}}+1\right)^q}{(R-r)^{N+2q+1}}.
	\end{align*}
	After iterating the above inequality, one can obtain the following inequality take the form similar to (\ref{12.3}), that is 
	\begin{align}\label{14.3}
		[u]_{W^{\gamma_i, q}\left(B_{\rho_i}\right)}^q \leq \frac{3^i\left(\prod\limits_{j=1}^i C_j\right) 2^{\left(N+2q+1+\frac{Nq}{p-1}\right)\left(\sum\limits_{j=1}^i j\right)}\left(\mathcal{T}+[u]_{W^{\gamma, q}\left(B_{R }\right)}+1\right)^q}{(R-r)^{(N+2 q+1) i}}
	\end{align}
	for any $i\in \mathbb{N}$.\par
	Now, we suppose $i_0$ be the smallest integer such that $\gamma_{i_0}\geq \rho_0$. In the case,
	\begin{align*}
		i_0:=\left[\frac{\ln \frac{1-\alpha}{2(\alpha-\gamma)}}{\ln \left(1-\frac{p^2}{2q}+\frac{p}{2q}\right)}\right]+2,
	\end{align*}
	and (\ref{14.3}) enables us to obtain
	\begin{align*}
		[u]_{W^{\alpha, q\left(B_r\right)}}^q \leq \frac{C\left(\mathcal{T}+[u]_{W^{\gamma,q}\left(B_R\right)}+1\right)^q}{(R-r)^{\kappa}},
	\end{align*}
	where $C:=\widetilde{C}_*^{i_0} 3^{i_0} 2^{\left(N+2q+1+\frac{Nq}{p-1}\right)\left(\sum\limits_{j=1}^{i_0} j\right)}$, $ \kappa:=(N+2q+1) i_0$ and $\widetilde{C}_*$ is a constant that satisfies for any $j$, 
	\begin{align*}
		C_j &\le \frac{\widetilde{C}\left(N, p, q, s_1, s_p\right)}{\gamma_j(1-\gamma_{j+1})^{q+1}(\min\left\lbrace \beta_i,1\right\rbrace-\gamma_i )^{\frac{q-p}{q}}}\le \frac{\widetilde{C}\left(N, p, q, s_1, s_p\right)}{\gamma(1-\alpha)^{q+2}}=:\widetilde{C}_*.
	\end{align*}
\end{proof}

With Lemma \ref{lem14} in hand, we can show that the weak solutions to problem \eqref{1.1} are almost Lipschitz continuous.

\begin{proof}[Proof of Theorem \ref{th26}]
	By (\ref{13.1}), we know that $u \in W_{\mathrm{loc}}^{\frac{s_p p}{q}, q}(\Omega)$, 
	which enables us to apply Lemma \ref{lem14} to obtain 
	\begin{align}\label{15.0}
		[u]_{W^{\alpha, q}\left(B_{\tilde{r}}\right) }^q & \leq \frac{\widetilde{C}\left(\mathcal{T}+[u]_{W^{\frac{s_p p}{q}, q}\left(B_R\right)}+1\right)^q}{(R-\tilde{r})^{\tilde{\kappa}} } \nonumber\\
		& \leq \frac{C\left(\mathcal{T}+[u]_{W^{s_ p, p}\left(B_R\right)}+1\right)^q}{(R-\tilde{r})^{\tilde{\kappa}}}
	\end{align}
	for any $\alpha\in(\frac{s_pp}{q},1)$ with $\tilde{r}=\frac{R+r}{2}$. We take $k:=\frac{1}{2}\left(1+\frac{1-s_p p}{2-p}\right)$, and apply (\ref{11.30}) under the assumption that $u \in W_{\mathrm{loc}}^{\alpha, q}(\Omega)$
	with $\sigma=k\alpha p$ to arrive at
	\begin{align}\label{15.1}
		\int_{B_r}\left|\tau_h\left(\tau_hu\right)\right|^q d x \leq \frac{C\left(\mathcal{T}+[u]_{W^{\alpha, q}\left(B_{\tilde{r}}\right)}+1\right)^q}{\left(\tilde{r}-r\right)^{N+2q+1}}|h|^{\alpha\left(q-\frac{p}{2}\right)+k\alpha p\left(1-\frac{p}{2}\right)+\frac{s_pp^2}{2}}.
	\end{align}
	Here we choose $\alpha=\frac{2q-s_p p^2}{4q-2p+4kp-2kp^2}+\frac{1}{2}$ to ensure 
	\begin{align*}
		\alpha\left(q-\frac{p}{2}\right)+k\alpha p\left(1-\frac{p}{2}\right)+\frac{s_p p^2}{2}>q .
	\end{align*}
	Therefore, Lemma \ref{lem6} and  (\ref{15.1}) enable us to conclude the following inequality
	\begin{align}\label{15.3}
		\int_{B_r}\left|\tau_h u\right|^q d x & \leq C|h|^q\left[\frac{C\left(\mathcal{T}+[u]_{W^{\alpha, q}(B_{\tilde{r}})}+1\right)^q}{\left(\tilde{r}-r\right)^{N+2q+1}}+\frac{\|u\|_{L^{\infty}\left(B_{\tilde{r}}\right)}^q}{\left(\frac{\tilde{r}-r}{2}\right)^q}\right]\nonumber\\
		& \leq \frac{C|h|^q\left(\mathcal{T}+[u]_{W^{\alpha, q}\left(B_r\right)}+1\right)^q}{(\tilde{r}-r)^{N+2q+1}}.
	\end{align}\par
	After a direct application of Lemma \ref{lem16}, one concludes that
	\begin{align}
		\int_{B_r}|\nabla u|^q d x \leq \frac{C\left(\mathcal{T}+[u]_{W^{\alpha, q}}\left(B_r\right)+1\right)^q}{(\tilde{r}-r)^{N+2q+1}}.
	\end{align}
	Bringing (\ref{15.0}) into (\ref{15.3}), we get 
	\begin{align*}
		\int_{B r}|\nabla u|^q d x&\leq \frac{C\left(\mathcal{T}+\frac{\left(\mathcal{T}+[u]_{W^{s_p,p}\left(B_R\right)}+1\right)}{\left(R-\tilde{r}\right)^{\frac{\tilde{\kappa}}{q}}}+1\right)^q}{\left(\tilde{r}-r\right)^{N+2q+1}}\\
		&\leq \frac{C\left(\mathcal{T}+[u]_{W^{s_p, p}\left(B_R\right)}+1\right)^q}{(R-r)^{\kappa}},
	\end{align*}
	where $\kappa:=\tilde{\kappa}+N+2 q+1$.
\end{proof}

\begin{proof}[Proof of Corollary \ref{cor27}]
	From Theorem \ref{th26}, we know that the locally bounded weak solution $u\in W^{1,q}_\mathrm{loc}(\Omega)$ for any $q\ge p$, provide that $p\in (1,2)$ and $s_p\in\left( \frac{p-1}{p},1\right) $. Thus, for any fixed $\gamma\in(0,1)$, we apply Lemma \ref{Morrey2} with the parameter $q:=\frac{N}{1-\gamma}$ to obtain the estimate
	\begin{align*}
		[u]_{C^{0,\gamma}(B_r)}= [u]_{C^{0,1-\frac{N}{q}}(B_r)}\le C\|\nabla u\|_{L^q(B_r)}\le \frac{C\left( \mathcal{T}+[u]_{W^{s_p,p}(B_R)}+1\right) }{(R-r)^{\kappa}}.
	\end{align*}
\end{proof}

Finally, we derive the Sobolev regularity of the gradient of weak solutions.
To this end, we first present a useful lemma that provides a new formulation of Proposition \ref{pro9}.
\begin{lemma}
	\label{lem20}
		Let $p\in(1,2)$, $s_1\in(0,1)$, $s_p\in \left( \frac{p-1}{p},1\right) $ and $q\ge p$. There exists a constant $C$ depending on $N,p,q,s_1,s_p,\sigma$ such that for any $u\in W^{1,q}_{\mathrm{loc}}(\Omega)$ being a locally bounded weak solution to problem \eqref{1.1} in the sense of Definition \ref{def1}, we have, for any $R\in(0,1)$, $r\in(0,R)$, $d\in\left(\frac{1}{4}(R-r) \right] $, $B_{R+d}\equiv B_{R+d}(x_0)\subset\subset\Omega$, $h\in \overline{B_d}\backslash\left\lbrace 0\right\rbrace $ and $\sigma\in(0,p)$,
	\begin{align*}
		&\quad\int_{B_r}\int_{B_r}\frac{\left| J_{\frac{q}{p}+1}\left( \tau_hu(x)\right)-J_{\frac{q}{p}+1}\left( \tau_hu(y)\right) \right|^p }{|x-y|^{N+\frac{s_pp^2}{2}+\sigma\left( 1-\frac{p}{2}\right) }}\,dxdy\\
		&\le \left( \frac{C}{(R-r)^{N+s_pp+1}}\right)^\frac{p}{2}\left[ \|\nabla u\|^q_{L^q(B_r)}\left( \int_{B_r}|\tau_hu|^q\,dx\right)^\frac{q-p}{q} \right]^{1-\frac{p}{2}}\\
		&\quad\times\left(\int_{B_R}|\tau_hu|^{q-p+1}\,dx+\int_{B_R}|\tau_hu|^q\,dx+|h|\mathcal{T}^{p-1}\int_{B_R}|\tau_hu|^{q-p+1}\,dx\right) ^\frac{p}{2}
	\end{align*}
	with
	\begin{align*}
		\mathcal{T}:=\|u\|_{L^\infty(B_{R+d}(x_0))}+\mathrm{Tail}(u;x_0,R+d).
	\end{align*}
\end{lemma}
\begin{proof}
	After a direct application of Lemma \ref{embed} to the inequality obtained in Proposition \ref{pro9}, we derive the desired result. More precisely, we set $\gamma:=\frac{p+\sigma}{2p}$ in Proposition \ref{pro9} and take \eqref{embedd} into consideration to have
	\begin{align*}
		&\quad\int_{B_r}\int_{B_r}\frac{\left| J_{\frac{q}{p}+1}\left( \tau_hu(x)\right)-J_{\frac{q}{p}+1}\left( \tau_hu(y)\right) \right|^p }{|x-y|^{N+\frac{s_pp^2}{2}+\sigma\left( 1-\frac{p}{2}\right) }}\,dxdy\\
		&\le \left( \frac{C}{(R-r)^{N+s_pp+1}}\right)^\frac{p}{2}\left[ [u]^p_{W^{\gamma,q}(B_r)}\left( \int_{B_r}|\tau_hu|^q\,dx\right)^\frac{q-p}{q} \right]^{1-\frac{p}{2}}\\
		&\quad\times\left(\int_{B_R}|\tau_hu|^{q-p+1}\,dx+\int_{B_R}|\tau_hu|^q\,dx+|h|\mathcal{T}^{p-1}\int_{B_R}|\tau_hu|^{q-p+1}\,dx\right) ^\frac{p}{2}\\
		&\le \left( \frac{C}{(R-r)^{N+s_pp+1}}\right)^\frac{p}{2}\left[\|\nabla u\|^q_{L^q(B_r)}\left( \int_{B_r}|\tau_hu|^q\,dx\right)^\frac{q-p}{q} \right]^{1-\frac{p}{2}}\\
		&\quad\times\left(\int_{B_R}|\tau_hu|^{q-p+1}\,dx+\int_{B_R}|\tau_hu|^q\,dx+|h|\mathcal{T}^{p-1}\int_{B_R}|\tau_hu|^{q-p+1}\,dx\right) ^\frac{p}{2},
	\end{align*}
	where $C$ takes the form of
	\begin{align*}
		C=\frac{C(N,p,q,s_1,s_p)}{(p-\sigma)^{\frac{q-p}{q}+\frac{p}{2}}}.
	\end{align*}
\end{proof}

Thanks to Lemma \ref{lem21}, which establishes the relationship between second-order differences and Sobolev regularity, we can show the regularity of the gradient of weak solutions.

\begin{proof}[Proof of Corollary \ref{corG}]
	We set $\widetilde{R}:=\frac{R+r}{2}$, $\tilde{r}:=\frac{1}{3}r+\frac{2}{3}\widetilde{R}$ and $d:=\frac{1}{6}(\widetilde{R}-r)$. Applying Lemmas \ref{lem16} and \ref{lem4} and following the method used in proof of Lemma \ref{lem11}, we can deduce  an inequality similar to \eqref{11.30}, that is,
	\begin{align*}
		\int_{B_{\tilde{r}}}\left|\tau_h\left(\tau_h u\right)\right|^q d x \leq \frac{C\left(\mathcal{T}+\|\nabla u\|_{L^q(B_{\widetilde{R}})}+1\right)^q}{(R-r)^{N+2q+1}}|h|^{\left(q-\frac{p}{2}\right)+\sigma\left(1-\frac{p}{2}\right)+\frac{s_p p^2}{2}}.
	\end{align*}
	Choosing $\sigma:=\frac{\left[\alpha+\frac{p}{2q}(s_pp-p+1)\right]+p-s_pp^2}{2-p}<p$ in the previous inequality, we have
	\begin{align*}
		\int_{B_{\tilde{r}}}\left|\tau_h\left(\tau_h u\right)\right|^q d x \leq \frac{C\left(\mathcal{T}+\|\nabla u\|_{L^q(B_{\widetilde{R}})}+1\right)^q}{(R-r)^{N+2q+1}}|h|^{\tilde{\alpha}q+q}
	\end{align*}
	with $\tilde{\alpha}:=\frac{\alpha+\frac{p}{2q}(s_pp-p+1)}{2}$. As a result of application of Lemma \ref{lem21}, we conclude
	\begin{align*}
		\nabla u\in \left( W^{\alpha,q}(B_r)\right)^N 
	\end{align*}
	for any $\alpha\in \left( 0,\frac{p}{2q}(s_pp-p+1)\right) $ and there exists a constant $C=C(N, p, q, s_1, s_p, \alpha)$ such that
	\begin{align*}
		[\nabla u]^{q}_{W^{\alpha,q}(B_r)}&\le C\left[ \frac{\left( \mathcal{T}+\|\nabla u\|_{L^q(B_{\widetilde{R}})}+1\right)^q}{(R-r)^{N+s_pp+1}}+\frac{\|\nabla u\|^q_{L^q(B_{\widetilde{R}})}}{(R-r)^{\left( \frac{p}{2q}(s_pp-p+1)+1\right) q}}\right] \\
		&\le \frac{C\left( \mathcal{T}+\|\nabla u\|_{L^q(B_{\widetilde{R}})}+1\right)^q}{(R-r)^{N+2q+2}}.
	\end{align*}
	By Theorem \ref{th26}, one can estimate the term $\|\nabla u\|^q_{L^q(B_{\widetilde{R}})}$ to deduce
	\begin{align*}
		[\nabla u]^{q}_{W^{\alpha,q}(B_r)}\le \frac{C\left( \mathcal{T}+[u]_{W^{s_p,p}(B_R)}+1\right)^q}{(R-r)^{\kappa}},
	\end{align*}
which completes the proof.
\end{proof}

\subsection*{Acknowledgments}
This work was supported by the National Natural Science Foundation of China (No. 12471128).

\subsection*{Conflict of interest} The authors declare that there is no conflict of interest. We further declare that this manuscript has no associated data.
	
\subsection*{Data availability}
No data was used for the research described in the article.

\end{document}